\documentclass[preprint,11pt]{elsarticle}
\usepackage{amssymb,amsmath,latexsym,amsfonts,array,graphicx,subfigure,mathrsfs,appendix,url,color}
\usepackage{epstopdf}
\usepackage{enumerate}
\usepackage{float}
\usepackage{leftidx}
\usepackage{booktabs}

\newtheorem{proposition}{Proposition}[section]
\newtheorem{theorem}[proposition]{Theorem}
\newtheorem{lemma}[proposition]{Lemma}
\newtheorem{corollary}[proposition]{Corollary}

\newtheorem{example}[proposition]{Example}

\newenvironment{proof}{{\noindent \bf Proof:}}{\hfill $\fbox{}$ \vspace*{5mm}}

\newcommand{\BE}{\begin{equation}}
\newcommand{\EE}{\end{equation}}

\begin{document}
\begin{frontmatter}
\title{Preconditioned iterative methods for space-time fractional advection-diffusion equations}


\author{Zhi Zhao}\ead{zhaozhi231@163.com}
\address{Department of Mathematics, University of Macau, Macao, China}
\author{Xiao-Qing Jin\fnref{fn1}}\ead{xqjin@umac.mo}
\address{Department of Mathematics, University of Macau, Macao, China}

\author{Matthew M. Lin\corref{cor2}\fnref{fn2}}\ead{mhlin@ccu.edu.tw}
\address{Department of Mathematics, National Chung Cheng University, Chia-Yi 621, Taiwan}

\cortext[cor2]{Corresponding author}

\fntext[fn1]{The second
author was supported by the research grant MYRG098(Y2-L3)-FST13-JXQ from University of Macau.
}
\fntext[fn2]{The third author was supported 
 by the Ministry of Science and Technology of Taiwan under grant
NSC101-2115-M-194-007-MY3.}

\date{}

\begin{abstract}
In this paper we want to propose  practical numerical methods to
solve a class of initial-boundary problem of space-time fractional
advection-diffusion equations. To start with, an implicit method
based on two-sided Gr\"{u}nwald formulae is proposed with a discussion
of the stability and consistency.  Then, the preconditioned generalized
minimal residual (preconditioned GMRES) method and the preconditioned
conjugate gradient normal residual ({preconditioned} CGNR) method,
with an easily constructed preconditioner, are developed. Importantly,
because the resulting systems are Topelitz-like, the fast Fourier transform
can be applied to significantly reduce the computational cost. Numerical
experiments are implemented to show the efficiency of our preconditioner,
even with cases of variable coefficients.

%
\end{abstract}


\end{frontmatter}

\section{Introduction}
%
%
%
%

This article is concerned with numerical approaches for solving the
initial-boundry value problem of the
space-time fractional advection-diffusion equation  (STFDE)~\cite{Liu2007}:

 \begin{equation}\label{FDEs}
 \left\{\begin{array}{l}
 \begin{array}{ll}
\dfrac{\partial^\alpha u(x,t)}{\partial t^\alpha}= & -d_+(x,t)
 D^\beta_{a,x} u(x,t)
 -
  d_-(x,t)
 D^\beta_{x,b} u(x,t) +\\
 & e_+(x,t) D^\gamma_{a,x} u(x,t)
  +e_-(x,t) D^\gamma_{x,b} u(x,t)
  + f(x,t),
  \end{array}
  \\
 u(x,0) = \phi(x),\quad a\leq x \leq b,
  \\
  u(a,t) = u(b,t) =0, \quad 0<t\leq T,
  \end{array}\right.
 \end{equation}
where $\alpha,\beta \in (0,1]$, $\gamma\in(1,2]$, $a<x<b$, and $0<t\leq T$.
Here, the parameters $\alpha, \beta$  and  $\gamma$ are the order of the STFDE,
$f(x,t)$ is the source term, and diffusion coefficient functions
$d_\pm(x,t)$ and $e_\pm(x,t)$ are non-negative under the assumption
that the flow is from left to right. The STFDE can be regarded as
generalizations of classical advection-diffusion equations with the
first-order time derivative replaced by the Caputo fractional derivative
of order $\alpha \in (0,1]$, and  the first-order and the second-order
space derivatives replaced by the two-sided Riemman-Liouville fractional
derivatives of order $\beta \in(0,1]$ and of order  $\gamma\in(1,2]$. Namely,
the time fractional derivative in~\eqref{FDEs}
is the Caputo fractional derivative of order $\alpha$~\cite{Podlubny1999}  denoted by
\begin{equation}
\frac{\partial^\alpha u(x,t)}{\partial t^\alpha} =
\frac{1}{\Gamma(1-\alpha)}
\int_0^t\frac{\partial u(x,\psi)}{\partial \psi} \frac{d\psi}{(t-\psi)^\alpha},
\end{equation}
and the left-handed ($D^\alpha_{a,x}$) and the right-handed ($D^\alpha_{x,b}$) space fractional derivatives in~\eqref{FDEs} are the Riemann-Liouville fractional derivatives of order $\alpha$~\cite{ Podlubny1999, Samko1993} which are defined by
\begin{subequations}\label{eq:RL}
 \begin{eqnarray}
 D^\alpha_{a,x} u(x,t)  &= &
\frac{1}{\Gamma(m-\alpha)}
\frac{\partial^m}{\partial x^m} \int^{x}_a  \frac{u(s,t)}{(x-s)^{\alpha-m+1}} ds,
 \end{eqnarray}
\mbox{and}
\begin{eqnarray}
 D^\alpha_{x,b} u(x,t)  &= &
\frac{(-1)^m}{\Gamma(m-\alpha)}
\frac{\partial^m}{\partial x^m} \int^{b}_x  \frac{u(s,t)}{(s-x)^{\alpha-m+1}} ds,
 \end{eqnarray}
where $\Gamma$ denotes the gamma function, and $m$ is an integer satisfying $m-1 < \alpha \leq m$.
\end{subequations}
Truly,  when $\alpha = \beta = 1$ and $\gamma =2$, the above equation reduces to the classical advection-diffusion equation.

The study of fractional calculus can be traced to late 17th century~\cite{Miller1993,Samko1993,Oldham1974},
but it was not until late 20th century that fractional differential equations (FDEs)
become important due to its wide applications in finance~\cite{Gorenflo2001,Raberto2002,Samko1993,Scalas2000}, physics~\cite{Atanackovic2004,Barkai2000,Carreras2001,Kilbas2006,Meerschaert2001,Meerschaert2002a,Meerschaert2002b},
image processing~\cite{Bai2007}, and even biology~\cite{West2007}.
Though analytic approaches, such as the Fourier transform method,
the Laplace transform methods, and the Mellin transform method,
have been proposed to seek closed-form solutions~\cite{Podlubny1999},
there are very few FDEs whose analytical closed-form solutions are available.
Therefore, the research on numerical approximation and techniques for the
solution of FDEs has attracted intensive interest; see~\cite{Ervin2007,Ford2001,Liu2004,Meerschaert2004,Meerschaert2006a,Meerschaert2006b,Sousa2009,Tadjeran2006,Tadjeran2007,Wang2010} and references therein.
Importantly, traditional methods for solving FDEs tend to generate full coefficient matrices, which incur computational cost of $\mathcal{O}(N^3)$ and storage of $\mathcal{O}(N^2)$ with $N$ being the number of grid points~\cite{Wang2010}.

To optimize the computational complexity, a shifted Gr\"{u}nwald discretization scheme with the property of unconditional stability was proposed by Meerschaet and Tadjeran~\cite{Meerschaert2004,Meerschaert2006a} to approximate the FDE. Later, Wang \emph{et al.}~\cite{Wang2010} discovered that the linear system generated by this discretization has a special Toeplitz-like coefficient matrix, or,	 more precisely,  this coefficient matrix can be expressed as a sum of diagonal-multiply-Toeplitz matrices.  This implies that the storage requirement is $\mathcal{O}(N)$ instead of $\mathcal{O}(N^2)$, and the complexity of the matrix-vector multiplication only requires $\mathcal{O}(N \log N)$ operations by the fast Fourier transform (FFT) \cite{Chan1996, Ng2004, Chan2007}. Upon using this advantage, Wang~\emph{et al.} proposed the CGNR method having computational cost of $\mathcal{O}(N\log^2 N)$ to solve the linear system and numerical experiments show that the CGNR method is fast when the diffusion coefficients are \text{very small}, i.e., the discretized systems are well-conditioned~\cite{Wang2011}.

However, the discretized systems become ill-conditioned when the diffusion coefficients are not small. In this case, the CGNR method converges slowly.
To overcome this shortcoming, preconditioning techniques have been introduced to improve the efficiency of the CG method with the total complexity being $\mathcal{O}(N \log N)$ operations at each time step~\cite{Lei2013,Lin2014}. For the same reason, we propose two preconditioned iterative methods, i.e., the preconditioned GMRES method and the preconditioned CGNR method, and observe results related to the acceleration of the convergence of the iterative methods, while solving~\eqref{FDEs}.

This paper is organized as follows. In section 2, we give an implicit difference method for~\eqref{FDEs} and prove that this scheme is unconditionally stable, convergent and uniquely solvable. In section 3, we propose the preconditioned GMRES method and the preconditioned CGNR method to solve~\eqref{FDEs} by exploring the matrix representation of the implicit difference scheme. Finally, we present numerical experiments to show the efficiency of our numerical approaches in section 4 and provide concluding remarks in section 5.

\section{Implicit difference method}

In this section, we present an implicit difference method for solving~\eqref{FDEs} by discretizing the STFDE defined by~\eqref{FDEs}. Unlike the approach given by Liu~\emph{et al.} in~\cite{Liu2007}, we use henceforth two-sided fractional derivatives to approximate the Riemann-Liouville derivatives in~\eqref{eq:RL}. We want to show that, by two-sided fractional derivatives, this method is also unconditionally stable and convergent.

\subsection{Discretization of the STFDE}

To start with, let $m$ and $n$ be two positive integers, and let  $h = {(b-a)}/{m}$ and $\tau = T/n$ be the sizes of time step and spatial grid, respectively. Then the spatial and temporal partitions can be defined by
\begin{equation*}
x_i =  a + i h, \quad i = 0,1,\ldots,m;\quad
t_j = j\Delta t, \quad j = 0,1,\ldots, n,
\end{equation*}
and for convenience, we shall denote henceforth
\begin{equation*}
\begin{array}{lll}
d^{(j)}_{+,i} = d_+(x_i,t_j),  & d^{(j)}_{-,i} = d_-(x_i,t_j),  &
e^{(j)}_{+,i} = e_+(x_i,t_j),  \\[2mm]
e^{(j)}_{-,i} = e_-(x_i,t_j), & f_i^{(j) } = f(x_i,t_j), &
\Delta_t u(x_i,t_{j }) = u(x_i,t_{j+1})-u(x_i, t_{j}).
\end{array}
\end{equation*}

Upon utilizing the forward difference formula, it is known that
the time fractional derivative for $0 < \alpha < 1$ can be approximated by~\cite{Liu2007},
\begin{eqnarray}
\frac{\partial^\alpha u(x_i,t_{k+1})}{\partial t^\alpha} &=&
\frac{1}{\Gamma(1-\alpha)}
\int_0^{t_{k+1}}\frac{\partial u(x_i,s)}{\partial s} \frac{ds}{(t_{k+1}-s)^\alpha} \nonumber\\
&=&
\frac{1}{\Gamma(1-\alpha)}
\sum_{j=0}^k \left (  \left(\frac{1}{\tau} \Delta_t u(x_i,t_{j}) + \mathcal{O}(\tau) \right)
\int_{t_j}^{t_{j+1}}  (t_{k+1}-s)^{-\alpha} ds\right )\nonumber + \mathcal{O}(\tau^{2-\alpha}) \\
&=&
\frac{\tau^{-\alpha}}{\Gamma(2-\alpha)}
\sum_{j=0}^k a_j \Delta_t u(x_i,t_{k-j}) + \mathcal{O}(\tau^{2-\alpha}), \label{alpha1}
\end{eqnarray}
where  $a_j = (j+1)^{1-\alpha} - j^{1-\alpha}$, $j = 0,1,\ldots, n$.
Also, the Riemann-Liouville derivatives in~\eqref{eq:RL} can be approximated by adopting the Gr$\rm{\ddot{u}}$nwald estimates
and the shifted Gr$\rm{\ddot{u}}$nwald estimates (see~\cite[Remark 2.5]{Meerschaert2004}) for parameters $\beta$ and $\gamma$, respectively, i.e.,
\begin{subequations}\label{eq:ARL}
\begin{eqnarray}
D^{(\beta)}_{a,x} u(x_i,t_{k+1}) &=& \frac{1}{h^\beta} \sum^i_{j=0} g_j^{(\beta)} u(x_{i-j},t_{k+1})
+ \mathcal{O}(h), \label{beta1}\\
D^{(\beta)}_{x,b} u(x_i,t_{k+1}) &=& \frac{1}{h^\beta} \sum^{m-i}_{j=0} g_j^{(\beta)} u(x_{i+j},t_{k+1})+ \mathcal{O}(h), \label{beta2}\\
D^{(\gamma)}_{a,x} u(x_i,t_{k+1}) &=& \frac{1}{h^\gamma} \sum^{i+1}_{j=0} g_j^{(\gamma)} u(x_{i-j+1},t_{k+1})
+ \mathcal{O}(h), \label{gamma1}\\
D^{(\gamma)}_{x,b} u(x_i,t_{k+1}) &=& \frac{1}{h^\gamma} \sum^{i+1}_{j=0} g_j^{(\gamma)} u(x_{i+j-1},t_{k+1})
+ \mathcal{O}(h), \label{gamma2}
\end{eqnarray}
\end{subequations}
where
\begin{eqnarray*}
g_0^{(\beta)}  &=& 1, \quad g_j^{(\beta)} = \frac{(-1)^j}{j!} \beta(\beta-1)\cdots(\beta-j+1), \quad j = 1,2,\ldots,\\
g_0^{(\gamma)}  &=& 1, \quad g_j^{(\gamma)} = \frac{(-1)^j}{j!} \gamma(\gamma-1)\cdots(\gamma-j+1), \quad j = 1,2,\ldots.
\end{eqnarray*}

Let
\begin{equation*}
\omega_1 = \frac{\Gamma(2-\alpha)\tau^\alpha}{h^\beta},\quad
\omega_2 = \frac{\Gamma(2-\alpha)\tau^\alpha}{h^\gamma}, \quad
\omega_3 = \Gamma(2-\alpha)\tau^\alpha,
\end{equation*}
and $u_i^{(j)}$ represent the numerical approximation of $u(x_i,t_j)$.
Using~\eqref{alpha1} and~\eqref{eq:ARL}, we shall see that
the solution of~\eqref{FDEs} can be approximated by the following \emph{implicit difference method}:


\begin{equation}\label{eq:alpha0}
\begin{array}{l}
u_i^{(k+1)}
 + \omega_1 \bigg (d_{+,i}^{(k+1)}  \sum\limits^i_{j=0} g_j^{(\beta)} u_{i-j}^{(k+1)}
 +  d_{-,i}^{(k+1)}  \sum\limits^{m-i}_{j=0} g_j^{(\beta)} u_{i+j}^{(k+1)}\bigg)
 - \omega_2  \bigg (e_{+,i}^{(k+1)} \sum\limits^{i+1}_{j=0} g_j^{(\gamma)} u_{i-j+1}^{(k+1)}
 \\[2mm]
 + e_{-,i}^{(k+1)} \sum\limits^{m-i+1}_{j=0} g_j^{(\gamma)} u_{i+j-1}^{(k+1)}\bigg)
 = u_i^{(k)} - \sum\limits_{j=1}^k a_j \big(u_i^{(k-j+1)} - u_i^{(k-j)}\big)
 + \omega_3 f_i^{(k+1)},
   \end{array}
\end{equation}
where $i = 1,\ldots, m-1$; $k=0,\ldots, n-1$, and the boundary and initial conditions can be discretized as follows:
\begin{equation*}
u_i^{(0)} =   \phi(x_i),  \quad i = 0,\ldots, m;\quad
u_0^{(k)} = u_m^{(k)} = 0, \quad  k = 1,\ldots,n.
 \end{equation*}

\subsection{Analysis of the implicit difference method}

To analyze the stability and convergence of the implicit difference method given above, we first let ${U}_i^{(k)}$ be the approximation solution of $u_i^{(k)}$ in~\eqref{eq:alpha0}, and let  $\xi_i^{(k)} = {U}_i^{(k)} - u_i^{(k)}$,
$i= 1,\ldots, m-1$; $k = 0,\ldots, n-1$, be the error satisfying the equation
\begin{equation}
\begin{array}{l}
\xi_i^{(k+1)}
 + \omega_1 \bigg (d_{+,i}^{(k+1)}  \sum\limits^i_{j=0} g_j^{(\beta)} \xi_{i-j}^{(k+1)}
 +  d_{-,i}^{(k+1)}  \sum\limits^{m-i}_{j=0} g_j^{(\beta)} \xi_{i+j}^{(k+1)}\bigg)
 - \omega_2  \bigg (e_{+,i}^{(k+1)} \sum\limits^{i+1}_{j=0} g_j^{(\gamma)} \xi_{i-j+1}^{(k+1)}
 \\[2mm]
 + e_{-,i}^{(k+1)} \sum\limits^{m-i+1}_{j=0} g_j^{(\gamma)} \xi_{i+j-1}^{(k+1)}\bigg)
 = \xi_i^{(k)} - \sum_{j=1}^{k}a_j\big(\xi_i^{(k-j+1)} -\xi_i^{(k-j)}\big).
 \end{array}
 \end{equation}

Correspondingly, assume $E^{(k+1)} = \big [ \xi_1^{(k)}, \xi_2^{(k)}, \ldots, \xi^{(k)}_{m-1} \big ]^\top$, $k = 0,\ldots,n-1$.
It is obvious upon inspection that the method given by~\eqref{eq:alpha0} is stable, once we can show that
\begin{equation*}
\|E^{(k+1)}\|_\infty \leq { \|E^{(0)}\|_\infty}.
\end{equation*}
To this purpose, the following results given in~\cite{Meerschaert2004,Meerschaert2006a,Wang2010} are required.
\begin{lemma}\label{propertyOfg}
The coefficients $a_j$, $g_j^{(\beta)}$, $g_j^{(\gamma)}$, for $j = 1,2,\ldots,$ satisfy
\begin{enumerate}
 \item 
 $ 1 = a_0 > a_1>a_2>\cdots>a_j\rightarrow 0$,
as $ j\rightarrow \infty$,

 \item $g_0^{(\beta)} = 1$, $g_j^{(\beta)} < 0$, for $j = 1,2,\ldots,$ and
 $\sum_{j=0}^\infty g_j^{(\beta)} = 0$,

 \item $g_1^{(\gamma)} = -\gamma < 0$, $g_j^{(\gamma)} > 0$, for $j \neq 1$, and $\sum_{j=0}^\infty g_j^{(\gamma)} = 0$.

 \item $g_j^{(\beta)} = \mathcal{O}(j^{-(\beta + 1)})$ and $ g_j^{(\gamma)} = \mathcal{O}(j^{-(\gamma + 1)})$.

 \end{enumerate}
  \end{lemma}

We do want to note that Lemma~\ref{propertyOfg} implies that
\begin{equation*}
\sum\limits_{j=0}^k g_j^{(\beta)} > 0 \mbox{ and }
 \sum\limits_{j=0}^{k+1} g_j^{(\gamma)} < 0, \quad \mbox{ for } k = 0,1, \ldots
\end{equation*}
This observation also gives rise to the certification of the stability of the method given by~\eqref{eq:alpha0}.
\begin{theorem}\label{thm:stability}
The implicit difference method~\eqref{eq:alpha0} for time-space fractional diffusion equation is unconditionally stable, that is,
\begin{equation}
\|E^{(k+1)}\|_\infty \leq { \|E^{(0)}\|_\infty},\quad 0\leq k\leq n-1.
\end{equation}

\end{theorem}

\begin{proof}
First, without loss of generality, we may assume that the diffusion coefficient functions $d_+(x,t) = d_+$, $d_-(x,t) = d_-$, $e_+(x,t) = e_+$ and $e_-(x,t) = e_-$ are constants in our proof. Suppose that $k=0$, and let $\xi_\ell^{(1)} = \|E^{(1)}\|_\infty:= \max\limits_{1\leq i\leq m-1} |\xi_i^{(1)}|$. Then
{\small
\begin{eqnarray*}
|\xi_\ell^{(1)}| &\leq&
\bigg[1
 + \omega_1 \Big (d_{+,\ell}^{(k+1)}  \sum\limits^\ell_{j=0} g_j^{(\beta)}
 +  d_{-,\ell}^{(k+1)}  \sum\limits^{m-\ell}_{j=0} g_j^{(\beta)} \Big)
-
 \omega_2  \Big (e_{+,\ell}^{(k+1)} \sum\limits^{\ell+1}_{j=0} g_j^{(\gamma)}
+ e_{-,\ell}^{(k+1)} \sum\limits^{m-\ell+1}_{j=0} g_j^{(\gamma)} \Big) \bigg] |\xi_\ell^{(1)}|\nonumber\\
&\leq&
|\xi_\ell^{(1)}| + \omega_1
\Big(
d_{+,\ell}^{(1)}  \sum^\ell_{j=0} g_j^{(\beta)} |\xi_{\ell-j}^{(1)}|
+
d_{-,\ell}^{(1)}  \sum^{m-\ell}_{j=0} g_j^{(\beta)} |\xi_{\ell+j}^{(1)} |
\Big)
-
 \omega_2 \Big(e_{+,\ell}^{(1)} \sum^{\ell+1}_{j=0} g_j^{(\gamma)}  |\xi_{\ell-j+1}^{(1)}|
 +
 e_{-,\ell}^{(1)} \sum^{m-\ell+1}_{j=0} g_j^{(\gamma)}  |\xi_{\ell+j-1}^{(1)}|
 \Big)
\nonumber\\
 & \leq &
 \bigg |\xi_\ell^{(1)} +\omega_1
\Big(
d_{+,\ell}^{(1)}  \sum^\ell_{j=0} g_j^{(\beta)} \xi_{\ell-j}^{(1)}
+
d_{-,\ell}^{(1)}  \sum^{m-\ell}_{j=0} g_j^{(\beta)} \xi_{\ell+j}^{(1)}
\Big)
-
 \omega_2 \Big(e_{+,\ell}^{(1)} \sum^{\ell+1}_{j=0} g_j^{(\gamma)}  \xi_{\ell-j+1}^{(1)}
 +
 e_{-,\ell}^{(1)} \sum^{m-\ell+1}_{j=0} g_j^{(\gamma)}  \xi_{\ell+j-1}^{(1)}
 \Big)
 \bigg |\nonumber\\
&=&
| \xi_\ell^{(0)} |
\leq \|E^{(0)}\|_\infty.
\end{eqnarray*}
}

Here, the second and third inequalities are true due to the fact given in Lemma~\ref{propertyOfg} and the triangle inequality on absolute value. Now suppose that for some integer $k\geq 0$, the result is established, i.e.,
\begin{equation*}
\|E^{(j)}\|_\infty \leq \|E^{(0)}\|_\infty, \quad \mbox{ for } j\leq k.
\end{equation*}
As we did earlier for $k = 0$, let $\xi_\ell^{(k+1)} = \max\limits_{1\leq i\leq m-1} |\xi_i^{(k+1)}|$.
By Lemma~\ref{propertyOfg}, it can be seen that

\begin{eqnarray*}
|\xi_\ell^{(k+1)}|
 & \leq &
 \bigg |\xi_\ell^{(k+1)} +\omega_1
\Big(
d_{+,\ell}^{(k+1)}  \sum^\ell_{j=0} g_j^{(\beta)} \xi_{\ell-j}^{(k+1)}
+
d_{-,\ell}^{(k+1)}  \sum^{m-\ell}_{j=0} g_j^{(\beta)} \xi_{\ell+j}^{(k+1)}
\Big)
\\
&&
-
 \omega_2 \Big(e_{+,\ell}^{(k+1)} \sum^{\ell+1}_{j=0} g_j^{(\gamma)}  \xi_{\ell-j+1}^{(k+1)} +
 e_{-,\ell}^{(k+1)} \sum^{m-\ell+1}_{j=0} g_j^{(\gamma)}  \xi_{\ell+j-1}^{(k+1)}
 \Big)
 \bigg |\nonumber\\
&=&
\bigg | \xi_\ell^{k}
 -\sum_{j=1}^k a_j
 \big(\xi_\ell^{k-j+1} - \xi_\ell^{k-j}\big) \bigg |
 =
\bigg | \sum_{j=1}^k (a_{j-1}-a_j)\xi_\ell^{(k-j+1)} +
a_k\xi_\ell^{(0)}   \bigg |
\\
&\leq&
\| E^{(0)}  \|_\infty.
\end{eqnarray*}

Truly, the preceding result, which follows from the assumption
that the coefficient functions are constant, does not provide complete results.
In fact, it can be seen that the above proof requires only the properties
of non-negatives of the coefficient functions. Thus, the result for non-constant ones can be proved similarly.

\end{proof}

Our next theorem is to analyze the convergence of the implicit method given in~\eqref{eq:alpha0}. To this end,
recall that $u(x_i,t_j)$, $i=1,\ldots,n-1;\,j=0,\ldots,n-1$, denotes the exact solution of~\eqref{FDEs} at mesh point $(x_i,t_j)$ and $u_i^{(j)}$, $i=1,\ldots,n-1;\,j=0,\ldots,n-1$,  represents the solution of~\eqref{eq:alpha0}.
Let us assume that $\psi_i^{(k)} = u(x_i,t_{k}) - u_i^{(k)} $ and $\Psi^{(k)} = (\psi_1^{(k)},\psi_2^{(k)},\ldots,\psi_{m-1}^{(k)})^\top$.
Note that, by construction, ${\Psi}^{(0)} = \mathbf{0}$, since $u_i^{(0)} = \psi(x_i)=u(x_i,0)$, $i = 1,\ldots,m-1$.

Using this notation, we consider
 \begin{equation}\label{eq:COV}
 \left\{
 \begin{array}{l}
\psi_i^{(1)}
 + \omega_1 \bigg (d_{+,i}^{(1)}  \sum\limits^i_{j=0} g_j^{(\beta)} \psi_{i-j}^{(1)}
 +  d_{-,i}^{(1)}  \sum\limits^{m-i}_{j=0} g_j^{(\beta)} \psi_{i+j}^{(1)}\bigg)
 -  \omega_2  \bigg (e_{+,i}^{(1)} \sum\limits^{i+1}_{j=0} g_j^{(\gamma)} \psi_{i-j+1}^{(1)}
 \\[2mm]
 + e_{-,i}^{(1)} \sum\limits^{m-i+1}_{j=0} g_j^{(\gamma)} \psi_{i+j-1}^{(1)}\bigg)
 = R_i^{(1)},
 \\[2mm]
\psi_i^{(k+1)}
 + \omega_1 \bigg (d_{+,i}^{(k+1)}  \sum\limits^i_{j=0} g_j^{(\beta)} \psi_{i-j}^{(k+1)}
 +  d_{-,i}^{(k+1)}  \sum\limits^{m-i}_{j=0} g_j^{(\beta)} \psi_{i+j}^{(k+1)}\bigg)
-
 \omega_2  \bigg (e_{+,i}^{(k+1)} \sum\limits^{i+1}_{j=0} g_j^{(\gamma)}
 \\[2mm]  \psi_{i-j+1}^{(k+1)}+
 e_{-,i}^{(k+1)} \sum\limits^{m-i+1}_{j=0} g_j^{(\gamma)} \psi_{i+j-1}^{(k+1)}\bigg)
 = \psi_i^{(k)} -  \sum_{j=1}^{k}a_j \big(\psi_i^{(k-j+1)} -\psi_i^{(k-j)}\big)+
 R_i^{(k+1)},\\ 
1\leq i\leq m-1,  \, 1\leq k\leq n-1.
 \end{array}
 \right.
\end{equation}

In this way, we can observe from~\eqref{alpha1} and~\eqref{eq:ARL} that
\begin{equation}\label{rbound}
 R_i^{(k+1)} = \mathcal{O}\big((\tau^{2} + \tau^{\alpha}h)\big),
 \quad
1\leq i\leq m-1;  \, 0\leq k\leq n-1. \\
\end{equation}
Thus, a way to do the convergence analysis is sufficed to come up with an upper bound of $\|\Psi^{(k+1)}\|_\infty$, $k = 0,1,\ldots,n-1$, as follows.
\begin{theorem}
\begin{equation}\label{eq:01}
\|\Psi^{(k+1)}\|_\infty \leq  {C}a_k^{-1}(\tau^{2} + \tau^{\alpha}h),\quad k = 0, \ldots, n-1,
\end{equation}
for some constant ${C}$.

\end{theorem}
\begin{proof}
Corresponding to~\eqref{rbound}, we shall assume for convenience that there is a positive constant $C$ such that
\begin{equation*}
 |R_i^{(k+1)}| \leq {C} (\tau^{2} + \tau^{\alpha}h),
 \quad 1\leq i\leq m-1;  \, 0\leq k\leq n-1.
\end{equation*}
Then, the poof is by mathematical induction on $k$.
Let $|\psi_\ell^1|= \|\Psi^1\|_\infty := \max\limits_{1\leq i\leq m-1} |\psi_i^1|$. Observe from~\eqref{eq:COV} that if $k=0$,  then we have
{\small
\begin{eqnarray*}
|\psi_\ell^{(1)}| &\leq&
\bigg[1
 + \omega_1 \Big (d_{+,\ell}^{(k+1)}  \sum\limits^\ell_{j=0} g_j^{(\beta)}
 +  d_{-,\ell}^{(k+1)}  \sum\limits^{m-\ell}_{j=0} g_j^{(\beta)} \Big)
-
 \omega_2  \Big (e_{+,\ell}^{(k+1)} \sum\limits^{\ell+1}_{j=0} g_j^{(\gamma)}
+ e_{-,\ell}^{(k+1)} \sum\limits^{m-\ell+1}_{j=0} g_j^{(\gamma)} \Big) \bigg] |\psi_\ell^{(1)}|\nonumber\\
&\leq&
|\psi_\ell^{(1)}| +\omega_1
\Big(
d_{+,\ell}^{(1)}  \sum^\ell_{j=0} g_j^{(\beta)} |\psi_{\ell-j}^{(1)}|
+
d_{-,\ell}^{(1)}  \sum^{m-\ell}_{j=0} g_j^{(\beta)} |\psi_{\ell+j}^{(1)} |
\Big)
-
 \omega_2 \Big(e_{+,\ell}^{(1)} \sum^{\ell+1}_{j=0} g_j^{(\gamma)}  |\psi_{\ell-j+1}^{(1)}|
 +
 e_{-,\ell}^{(1)} \sum^{m-\ell+1}_{j=0} g_j^{(\gamma)}  |\psi_{\ell+j-1}^{(1)}|
 \Big)
\nonumber\\
 & \leq &
 \bigg |\psi_\ell^{(1)} +\omega_1
\Big(d_{+,\ell}^{(1)}  \sum^\ell_{j=0} g_j^{(\beta)} \psi_{\ell-j}^{(1)}
+ d_{-,\ell}^{(1)}  \sum^{m-\ell}_{j=0} g_j^{(\beta)} \psi_{\ell+j}^{(1)}\Big)
-\omega_2 \Big(e_{+,\ell}^{(1)} \sum^{\ell+1}_{j=0} g_j^{(\gamma)}  \psi_{\ell-j+1}^{(1)}
 + e_{-,\ell}^{(1)} \sum^{m-\ell+1}_{j=0} g_j^{(\gamma)}  \psi_{\ell+j-1}^{(1)} \Big)
 \bigg |\nonumber\\
&=&
 |R_\ell^{(1)}|\leq {C} a_0^{-1}(\tau^{2} + \tau^{\alpha}h),
\end{eqnarray*}
}
namely,
 \begin{equation*}
\|\Psi^1\|_\infty \leq {C} a_0^{-1} (\tau^{2} + \tau^{\alpha}h),
\end{equation*}

Suppose that  the result is valid for some integer $k\geq 0$, i.e.,
\begin{equation}
\|\Psi^j\|_\infty \leq {C} a_{k-1}^{-1} (\tau^{2} + \tau^{\alpha}h),\quad j = 1, \ldots, k-1.
\end{equation}

Let $|\psi_\ell^{k+1}|  =\|\Psi^{k+1}\|_\infty:= \max\limits_{1\leq i\leq m-1} |\psi_i^{k+1}|$. It follows that
\begin{eqnarray*}
|\psi_\ell^{k+1}|
 & \leq &
 \bigg |\psi_\ell^{(k+1)} +\omega_1
\Big(
d_{+,\ell}^{(k+1)}  \sum^\ell_{j=0} g_j^{(\beta)} \psi_{\ell-j}^{(k+1)}
+
d_{-,\ell}^{(k+1)}  \sum^{m-\ell}_{j=0} g_j^{(\beta)} \psi_{\ell+j}^{(k+1)}
\Big)
\\
&&
-
 \omega_2 \Big(e_{+,\ell}^{(k+1)} \sum^{\ell+1}_{j=0} g_j^{(\gamma)}  \psi_{\ell-j+1}^{(k+1)} +
 e_{-,\ell}^{(k+1)} \sum^{m-\ell+1}_{j=0} g_j^{(\gamma)}  \psi_{\ell+j-1}^{(k+1)}
 \Big)
 \bigg |\nonumber\\
&=&
\bigg | \psi_\ell^{k}
 -\sum_{j=1}^k a_j\big (\psi_\ell^{k-j+1} - \psi_\ell^{k-j}\big) \bigg |
 =
\bigg | \sum_{j=1}^k (a_{j-1}-a_j)\psi_\ell^{(k-j+1)} +
a_k\psi_\ell^{(0)} +  R_\ell^{(k+1)}   \bigg |
 \\
 &\leq&
  \sum_{j=1}^k (a_{j-1}-a_j) \Big| \psi_\ell^{(k-j+1)} \Big| +  \Big|R_\ell^{(k+1)} \Big|   \\
 &\leq&
  {C} \bigg( a_k +  \sum_{j=1}^k (a_{j-1}-a_j)  \bigg)
a_{k}^{-1}(\tau^{2} + \tau^{\alpha}h)
\leq
 {C} a_k^{-1}(\tau^{2} + \tau^{\alpha}h),\nonumber
 \end{eqnarray*}
since $a_{j-1}- a_j > 0$, $j = 1,\ldots,k$, and $\psi_\ell^{(0)} = 0$.
\end{proof}

It has been shown in~\cite{Liu2007} that
\begin{equation}\label{eq:11}
\lim_{k\to \infty}\frac{a_k^{-1}}{k^{\alpha}}
= \frac{1}{1-\alpha}.
\end{equation}
By (\ref{eq:01}) and (\ref{eq:11}), we immediately have the following result, which demonstrates the convergence of our implicit method.
%
\begin{corollary}
Let $u_i^{(k)}$, $i=1,\ldots,m-1;\,k=1,\ldots,n$ be the numerical solution computed by the implicit difference method \eqref{eq:alpha0}.
Then, there exists a constant $C$ such that
\begin{equation}
| u(x_i,t_{k}) - u_i^{(k)} | \leq  C(\tau^{2-\alpha} + h),\quad i=1,\ldots, m-1;\,k=1,\ldots,n.
\end{equation}
\end{corollary}

We remark that the above approach used to analyze the stability and convergence is simply a follow-up used by Liu~\emph{et al.} in \cite{Liu2007}. Our focus in this work is to apply the efficient CGNR method and GMRES method to solve the linear system arised from~\eqref{eq:alpha0} in terms of suitably constructed preconditioners.

%
%
%
%
%
%
%

\section{Preconditioned iterative methods}

%

Before moving into the investigation of preconditioning techniques, the matrix representation of~\eqref{eq:alpha0} should be elaborated first.  To facilitate our discussion, we use $I_{m-1}$ to denote the identity matrix of order $m-1$. For $1\leq j \leq  n-1$, let
\begin{equation*}
\begin{array}{ll}
\mathbf{u}^{(j)} = [u_1^{(j)},u_2^{(j)},\ldots,u_{m-1}^{j}]^\top, &
\mathbf{f}^{(j)} = [f_1^{(j)}, f_2^{(j)},\ldots,f_{m-1}^{(j)}]^\top,\\[2mm]
 D_+^{(j)} = {\rm diag}(d_{+,1}^{(j)},...,d_{+,m-1}^{(j)}),&
  D_-^{(j)} = {\rm diag}(d_{-,1}^{(j)},...,d_{-,m-1}^{(j)}),\\[2mm]
E_+^{(j)} = {\rm diag}(e_{+,1}^{(j)},...,e_{+,m-1}^{(j)}),&
 E_-^{(j)} = {\rm diag}(e_{-,1}^{(j)},...,e_{-,m-1}^{(j)}),
\end{array}
\end{equation*}
and
$
\mathbf{u}^{(0)} = (\phi_1^{(0)},\phi_2^{(0)},\cdots,\phi_{m-1}^{(0)})^\top.
$
Let $G_\beta$ and $G_\gamma$ be two Toeplitz matrices defined by
\begin{equation*}
G_\beta = \left[\begin{array}{ccccc}
g_0^{(\beta)} & 0 & \cdots & \cdots & 0 \\
g_1^{(\beta)} & g_0^{(\beta)} & 0 & \cdots & 0 \\
\vdots & g_1^{(\beta)} & g_0^{(\beta)} & \ddots & \vdots \\
\vdots & \ddots & \ddots & \ddots & 0  \\
g_{m-2}^{(\beta)} & \ddots& \ddots & \ddots &g_{0}^{(\beta)}  \\
\end{array}\right],
\quad
G_\gamma = \left[\begin{array}{cccccc}
g_1^{(\gamma)} & g_0^{(\gamma)} & 0 & \cdots & 0 & 0 \\
g_2^{(\gamma)} & g_1^{(\gamma)} & g_0^{(\gamma)} & 0 & \cdots & 0 \\
\vdots & g_2^{(\gamma)} & g_1^{(\gamma)} & \ddots & \ddots & \vdots \\
\vdots & \ddots & \ddots & \ddots & \ddots & 0 \\
g_{m-2}^{(\gamma)} & \ddots& \ddots & \ddots &g_{1}^{(\gamma)} & g_{0}^{(\gamma)} \\
g_{m-1}^{(\gamma)} & g_{m-2}^{(\gamma)} & \cdots & \cdots & g_{2}^{(\gamma)} & g_{1}^{(\gamma)}
\end{array}\right].
\end{equation*}

Upon substitution, we see that~\eqref{eq:alpha0} is equivalent to a matrix equation of the form
\begin{eqnarray}\label{eq:FD2}
(I_{m-1} +  A^{(k+1)}) \mathbf{u}^{(k+1)} = \mathbf{b}^{(k+1)},
\end{eqnarray}
where
\begin{eqnarray*}\label{eq:Ab}
\mathbf{b}^{(k+1)} &=&  \sum_{j=1}^{k} (a_{k-j}-a_{k-j+1})\mathbf{u}^{(j)}+
a_k \mathbf{u}^{(0)} +
 \omega_3 \mathbf{f}^{k+1}
 \end{eqnarray*}
 and
 \begin{eqnarray}\label{eq:Ak}
A^{(k+1)} &=& \omega_1 (D_+^{(k+1)} G_\beta + D_-^{(k+1)} G_\beta^\top) -  \omega_2
(E_+^{(k+1)} G_\gamma + E_-^{(k+1)} G_\gamma^\top). \label{eq:Ab2}
\end{eqnarray}

Now we can define the corresponding matrix equation of~\eqref{eq:alpha0}. An intuitive question
to ask is whether the matrix equation is uniquely solvable. Before answering this,
we make an interesting observation of the following result.

\begin{theorem}\label{thm:Mmatrix}
The matrix $I_{m-1} + A^{(k+1)}$ in~\eqref{eq:FD2} is a nonsingular, strictly diagonally dominant $M$-matrix.
\end{theorem}

\begin{proof}
Let $a_{ij}^{(k+1)}$ be the $(i,j)$ entry of the matrix $A^{(k+1)}$ in~\eqref{eq:FD2}. Note that we have from~\eqref{eq:FD2},
\begin{equation}
\begin{array}{l}
a_{ii}^{(k+1)} - \sum\limits_{j=1,j\neq i}^{m-1} |a_{ij}^{(k+1)}| \\
=
\omega_1 \left (d_{+,i}^{(k+1)}
 +  d_{-,i}^{(k+1)} \right)   g_0^{(\beta)}
-
\omega_1 \bigg (d_{+,i}^{(k+1)}  \sum\limits^{i-1}_{j=1} g_j^{(\beta)}
 +  d_{-,i}^{(k+1)}  \sum\limits^{m-i-1}_{j=1} g_j^{(\beta)}\bigg)
 \\
 -
 \omega_2  \left (e_{+,i}^{(k+1)}+ e_{-,i}^{(k+1)} \right)g_1^{(\gamma)}
-
 \omega_2  \bigg (e_{+,i}^{(k+1)} \sum\limits^{i}_{j=0,j\neq 1} g_j^{(\gamma)} +  e_{-,i}^{(k+1)} \sum\limits^{m-i}_{j=0,j\neq 1} g_j^{(\gamma)} \bigg)
\\
\geq
 \omega_1 \left (d_{+,i}^{(k+1)}
 +  d_{-,i}^{(k+1)} \right)   g_0^{(\beta)}
-
\omega_1 \left (d_{+,i}^{(k+1)}  +  d_{-,i}^{(k+1)} \right) \sum\limits^{\infty}_{j=1} g_j^{(\beta)}
 \\
 -
 \omega_2  \left (e_{+,i}^{(k+1)}+ e_{-,i}^{(k+1)} \right)g_1^{(\gamma)}
-
 \omega_2  \left (e_{+,i}^{(k+1)} +  e_{-,i}^{(k+1)} \right) \sum\limits^{\infty}_{j=0,j\neq 1} g_j^{(\gamma)} = 0.
\end{array}
\end{equation}
At first glance,  this implies that the coefficient matrix $I_{m-1} + A^{(k+1)}$ is strictly diagonally dominant and $(I_{m-1}+A^{(k+1)}) \mathbf{1} > 0$, where $\mathbf{1}$ is a vector of length $n-1$ with all entries equal to one. We observe further that $a_{i,j}\leq 0$, for all $i\neq j$,  that is, the matrix $I_{m-1} + A^{(k+1)}$ is a $Z$-matrix.  This completes the proof.
\end{proof}

With the aid of Lemma~\ref{thm:Mmatrix}, we can point out quickly that the solution of~\eqref{eq:FD2} is unique.
More significantly, since~\eqref{eq:FD2} is a matrix representation of~\eqref{eq:alpha0},
we then come up with the following result.

\begin{corollary}\label{cor:uni2}
The difference method~\eqref{eq:alpha0} is uniquely solvable.
\end{corollary}

By now, we have completed the proof of the unique solvability of the implicit difference scheme given in~\eqref{eq:alpha0}.
We are now ready to apply the popular and effective iterative methods, the CGNR and GMRES methods, to solve~\eqref{eq:FD2}. In section~\ref{numerical}, we will see that while solving large-scale equations, the systems would become nearly singular and ill-conditioned. For such problems, we apply the preconditioner
technique to accelerate the iterative process.
%
%
%
%

To this purpose, we start by decomposing matrices $G_\beta$ and $G_\gamma$ as
\begin{eqnarray*}
G_\beta &=& G_{\beta,\ell} + (G_\beta -G_{\beta,\ell}),\\[2mm]
G_\gamma &=& G_{\gamma,\ell} + (G_\gamma -G_{\gamma,\ell}),
\end{eqnarray*}
where
\begin{eqnarray*}
G_{\beta,\ell} &=& \left[\begin{array}{ccccc}
g_0^{(\beta)} &  &  & \\
\vdots &g_0^{(\beta)} &   &  \\
g_{\ell-1}^{(\beta)} &  & \ddots &  \\
  & \ddots &  & \ddots &  \\
  &   & g_{\ell-1}^{(\beta)}& \cdots & g_0^{(\alpha)}
\end{array}\right] +
\left[\begin{array}{cccccc}
0 &  &  && \\
 &\ddots &  &  &\\
 &  & 0 & & \\
  &  &  & g_{\ell}^{(\beta)} &  & \\
  &   &  &  & \ddots&\\
    &   &  &  &  & \sum_{j=\ell}^{m-2} g_j^{(\beta)}
  \end{array}\right], \\[2mm]
  G_{\gamma,\ell} &=& \left[\begin{array}{ccccc}
g_1^{(\gamma)} & g_0^{(\gamma)} &  & \\
\vdots &g_1^{(\gamma)} & g_0^{(\gamma)} &  \\
g_{\ell}^{(\gamma)} &  & \ddots & \ddots \\
  & \ddots &  & \ddots & g_0^{(\alpha)} \\
  &   & g_{\ell}^{(\gamma)}& \cdots & g_1^{(\gamma)}
\end{array}\right] +
\left[\begin{array}{cccccc}
0 &  &  && \\
 &\ddots &  &  &\\
 &  & 0 & & \\
  &  &  & g_{\ell+1}^{(\gamma)} &  & \\
  &   &  &  & \ddots&\\
    &   &  &  &  & \sum_{j=\ell+1}^{m-1} g_j^{(\gamma)}
  \end{array}\right].
\end{eqnarray*}
Namely, the matrix $A^{(k+1)}$ can be decomposed as
\begin{equation*}
A^{(k+1)} = A^{(k+1)}_\ell + B^{(k+1)}_\ell,
\end{equation*}
where
\begin{eqnarray*}
A^{(k+1)}_\ell&=& \omega_1( D_+^{k+1} G_{\beta,\ell} + D_-^{k+1} G_{\beta,\ell}^\top) -\omega_2
(E_+^{k+1} G_{\gamma,\ell} + E_-^{k+1} G_{\gamma,\ell}^\top),\\[2mm]
B^{(k+1)}_\ell &=& A^{(k+1)} - A^{(k+1)}_\ell .
\end{eqnarray*}

Note that from Lemma~\ref{propertyOfg}, it is easy to show that the Toeplitz matrices $G_\beta$ and $-G_\gamma $ are $M$-matrices and
strictly diagonally dominant. This implies that the matrices $G_{\beta,\ell}$ and $G_{\gamma,\ell}$ are thus strictly diagonally dominant
$M$-matrices, since the matrices $G_{\beta,\ell}$ and $G_{\gamma,\ell}$ have the same row sums as $G_{\beta}$
and $G_{\gamma}$, respectively. In this way, the following fact can be realized directly.

\begin{theorem}\label{thm:propAk}
The matrix $I_{m-1} + A_\ell^{(k+1)}$ is a nonsingular, strictly diagonally dominant  M-matrix for all $\ell$.
\end{theorem}

In addition, Lemma~\ref{propertyOfg} implies that
{\small
\begin{eqnarray*}
&&\displaystyle \frac{ \| (I_{m-1} + A^{(k+1)})-(I_{m-1} + A^{(k+1)}_\ell) \|_\infty}{ \| I_{m-1} - A^{(k+1)} \|_\infty} \\[2mm]
&&
\leq \frac{ \frac{1}{h^{\gamma}}\|
( D_+^{k+1} (G_\beta - G_{\beta,\ell}) + D_-^{k+1} (G_\beta-G_{\beta,\ell})^\top) -
(E_+^{k+1} (G_\gamma - G_{\gamma,\ell}) + E_-^{k+1} (G_\gamma-G_{\gamma,\ell})^\top)
\|_\infty}{
\frac{1}{h^{\beta}}
\|  (D_+^{(k+1)} G_\beta + D_-^{(k+1)} G_\beta^\top) -
(E_+^{(k+1)} G_\gamma + E_-^{(k+1)} G_\gamma^\top) \|_\infty } =  O\big( k^{-\beta}\big),
\end{eqnarray*}
}
since $\|G_\beta - G_{\beta,\ell}\|_\infty = \mathcal{O}(k^{-\beta})$, $\|G_\gamma - G_{\gamma,\ell}\|_\infty = \mathcal{O}(k^{-\gamma})$, and $h= (b-a)/m$~\cite{Lin2014}. Namely, the relative difference between $I_{m-1} + A^{(k+1)}$ and $I_{m-1} + A_\ell^{(k+1)}$
can become very small while $k$ becomes large enough.
Observe further that
the banded matrix $I_{m-1} + A_\ell^{(k+1)}$ is a sparse matrix consisting of $2\ell-1$ nonzero diagonal entries.
With this in hand, an efficient precoditioner for the linear system~\eqref{eq:FD2} is attainable by simply choosing $I_{m-1} + A_\ell^{(k+1)}$. We assume here that the reader is familiar with the fundamental terminology and iterative approaches of the preconditioned GMRES and CGNR methods. For a comprehensive understanding of such iterative techniques, the reader is referred to the monograph \cite{Saad2003} written by Saad.

\subsection{Preconditioned GMRES method}

The GMRES method, proposed in 1986 in~\cite{Saad1986},
is one of the most popular and effective methods for solving nonsymmetric linear systems.
However, for large sparse systems, one might try to apply preconditioning techniques
to reduce the condition number, and hence improve the convergence rate.
Let $P^{(k+1)}_\ell:= I_{m-1} + A_\ell^{(k+1)}$. Our purpose here is to
replace the linear system~\eqref{eq:FD2} by the preconditioned linear system
\begin{eqnarray}\label{eq:GMRESFD2}
(P^{(k+1)}_\ell)^{-1}(I_{m-1} +  A^{(k+1)}) \mathbf{u}^{(k+1)} = (P^{(k+1)}_\ell)^{-1}\mathbf{b}^{(k+1)}
\end{eqnarray}
with the same solution.
We then solve~\eqref{eq:GMRESFD2} in terms of the left-preconditioned GMRES method proposed in~\cite{Lin2014}. 
To make this work more self-contained, we quote this method as follows:

\noindent\rule[0.1\baselineskip]{\textwidth}{1pt}\\[-0.2em]
\rm{Preconditioned GMRES($\rho$) method} \\[-0.6em]
\noindent\rule[0.1\baselineskip]{\textwidth}{1pt}

At each time step $t^{(k+1)}$, we choose $\mathbf{u}_0$
as initial guess  for $\mathbf{u}^{(k+1)}$

Set $\mu :=0$, and compute the LU factorization:
$P_\ell^{(k+1)} = LU$

Compute $\mathbf{r} := \mathbf{b}^{(k+1)} - (I_{m-1} + A^{(k+1)}) \mathbf{u}^{(k+1)} $,
and assign $\mathbf{r}_t := \mathbf{r}$

While $\mu \leq IterMax$ and $\| \mathbf{r}_t \|_2/\|\mathbf{b}^{(k+1)}\|_2 > \epsilon$ do

\quad $\mu:=\mu +1$

\quad Compute $\mathbf{r}_w := U^{-1}L^{-1}\mathbf{r}$, $\beta:=\| \mathbf{r}_w \|$,
$\mathbf{v}_1:= \mathbf{r}_w/\beta$

\quad  Assign $j:=0$ and $V_1:= \mathbf{v}_1$

\quad  While $j\leq \rho$ and $\| \mathbf{r}_t \|_2/\|\mathbf{b}^{(k+1)}\|_2 > \epsilon$ do

\qquad  $j:=j+1$

\qquad  Compute $\mathbf{w} :=U^{-1}L^{-1}(I_{m-1} + A^{(k+1)})\mathbf{v}_j$

\qquad  For $i=1,\ldots,j$ do

\qquad \quad  $h_{i,j}=\mathbf{v}^T_i\ast \mathbf{w}$

\qquad \quad  $\mathbf{w}:=\mathbf{w} - h_{i,j}\mathbf{v}_i$

\qquad  Enddo

\qquad  Compute $h_{j+1,j}=\|\mathbf{w}\|_2$ and $\mathbf{v}_{j+1}:=\mathbf{w}/h_{j+1,j}$

\qquad  Assign  $V_{j+1}:= [ V_{j},\mathbf{v}_{j+1}]$ and
${H_j}:=[h_{\gamma,\delta}]_{1\leq \gamma \leq j+1,1\leq \delta \leq j}$

\qquad  Compute $\mathbf{y}_j:= \mbox{argmin}_{\mathbf{y}}\|\beta \mathbf{e}_1
- {H_j}\mathbf{y}\|_2$

\qquad  Compute the residual $\mathbf{r}_t := \mathbf{r} - LUV_{j+1} {H_j}\mathbf{y}_j$

\quad  Enddo

\quad  $\mathbf{r}:=\mathbf{r}_t$

\quad  $\mathbf{u}^{(k+1)} := \mathbf{u}^{(k+1)}+ V_{j}\mathbf{y}_j$

Enddo\\[-0.3em]
\noindent\rule[0.15\baselineskip]{\textwidth}{1pt}

Here, $IterMax$ denotes the maximal number of iteration,
$\epsilon$ denotes the given relative accuracy of the residual, $\rho$ denotes that the GMRES method is restarted after $\rho$ iterations, and the symbols $\mathbf{r}_t$ and $\mathbf{r}_w$ represent the current residual of the original linear system~\eqref{eq:FD2} and that of the preconditioned linear system~\eqref{eq:GMRESFD2}, accordingly. Associated with this preconditioned method,  two major portions of the computational work are:
 \begin{itemize}
\item  the computation of $\mathbf{w}=U^{-1}L^{-1}(I_{m-1} + A^{(k+1)})\mathbf{v}_j$ and
\item  the computation of $\mathbf{r}_t= \mathbf{r}- LUV_{j+1} H_j \mathbf{y}_j$.
\end{itemize}
We observe from~\eqref{eq:Ak} that
\begin{eqnarray*}
A^{(k+1)}\mathbf{v} & =
 \omega_1 (D_+^{(k+1)} G_\beta \mathbf{v} + D_-^{(k+1)} G_\beta^\top \mathbf{v}) -  \omega_2
(E_+^{(k+1)} G_\gamma \mathbf{v} + E_-^{(k+1)} G_\gamma^\top \mathbf{v}),
\end{eqnarray*}
where  $G_\gamma$ and $G_\beta$ are two $(m-1)-$by$-(m-1)$ Toeplize matrices
and can be stored only with $m-1$ and $m$ entries, respectively. This implies that the major work for computing $A^{(k+1)}\mathbf{v}$ includes
four Toeplitz matrix-vector multiplications,
$G_\beta \mathbf{v}$, $G_\beta^\top \mathbf{v}$, $G_\gamma \mathbf{v}$ and $G_\gamma^\top \mathbf{v}$, which can be obtained by using the fast Fourier transform (FFT) with only $\mathcal{O}((m-1)\log (m-1))$ operations~\cite{Chan1996, Chan2007, Pang2012}. What might be important to note is that based on the specific structure of the matrix $G_s$, where $s = \beta$ or $\gamma$,  the calculations of
$G_s \mathbf{v}$ and $G_s^\top \mathbf{v}$ can be done simultaneously, by computing
$G_\beta(\mathbf{v} + \sqrt{-1}\hat{\mathbf{v}})$,
where $\hat{\mathbf{v}}=(v_{m-1},v_{m-2},\ldots,v_1)^\top$.

Since the matrix $P^{(k+1)}_\ell$  is banded and strongly diagonally dominant,  $P^{(k+1)}_\ell$ admits a banded $LU$ factorization~\cite[Proposition2.3]{Demmel1997}, i.e.,
\begin{equation}\label{eq:LU}
P^{(k+1)}_\ell = LU,
\end{equation}
where $L$ and $U$ are banded with bandwidth {$\ell$}
and can be obtain in about $\mathcal{O}((m-1)\ell^2)$ operations when $\ell$
is small compared to $(m-1)$.
This implies that given a vector $\mathbf{x}$ of an appropriate size,
the matrix-vector multiplications $L\mathbf{x}$, $U\mathbf{x}$, $L^{-1}\mathbf{x}$, and $U^{-1}\mathbf{x}$ require only $\mathcal{O}((m-1)\ell)$ operations. Thus, the computation of the vector $\mathbf{w}$
requires $\mathcal{O}((m-1)\log (m-1))$ operations, and
the computation of the vector $\mathbf{r}_t$  requires $\mathcal{O}((m-1)(j+\ell))$
operations since $V_{j+1}$ and $H_j$ are matrices of sizes $(m-1)-$by$-(j+1)$ and
$(j+1)-$by$-j$.

\subsection{Preconditioned CGNR method}

For solving the nonsymmetric linear system~\eqref{eq:FD2}, one might consider the application of the conjugate gradient (CG) method to the normal equation
\begin{equation}\label{eq:normal}
(I_{m-1} + A^{(k+1)})^\top (I_{m-1} + A^{(k+1)}) \mathbf{u}^{(k+1)} = (I_{m-1} + A^{(k+1)})^\top \mathbf{b}^{(k+1)}.
\end{equation}
This approach is known as CGNR. One disadvantage of applying the CG method directly to the equation~\eqref{eq:normal} is that the condition number of $(I_{m-1} + A^{(k+1)})^\top (I_{m-1} + A^{(k+1)})$ is the square of that of $I_{m-1} + A^{(k+1)}$. Thus, the convergence process
of the CGNR method would be very slow. To accelerate the entire process, we choose $(P^{(k+1)}_\ell)^\top P^{(k+1)}_\ell$ as
the preconditioner for the normal equation (\ref{eq:normal}).

Note that the main computational works in the preconditioned CGNR method include two parts~\cite{Saad2003}. One is the matrix-vector multiplication $(I_{m-1} + A^{(k+1)})^\top(I_{m-1} + A^{(k+1)}) \mathbf{v}$ for some vector $\mathbf{v}$. The other is the calculation of the solution of the linear system $\big(P^{(k+1)}_\ell\big)^\top
P^{(k+1)}_\ell \mathbf{w} = \mathbf{z} $ for some vectors $\mathbf{w}$ and $\mathbf{z}$.  Of course, like the preconditioned GMRES method, the calculation of the matrix-vector multiplication
$(I_{m-1} + A^{(k+1)})^\top (I_{m-1} + A^{(k+1)})\mathbf{v}$ can be done efficiently
by applying the fast algorithm, FFT, to the {Toeplitz-like} structure of the resulting matrix $A^{(k+1)}$ with $\mathcal{O}((m-1)\log(m-1))$ operations. Similarly, from~\eqref{eq:LU}, we know that the solution of $\big(P^{(k+1)}_\ell\big)^\top
P^{(k+1)}_\ell \mathbf{w} = \mathbf{z} $ can be obtained with only
$\mathcal{O}((m-1)\ell)$ operations.

\section{Numerical experiments}\label{numerical}

In this section, we present an example to demonstrate the performance
of preconditioned iterative methods versus unconditioned iterative methods.
 For all methods, the initial values are chosen to be
\[
\mathbf{v}_0 =
\left\{
\begin{array}{ll}
\mathbf{u}^{(0)} :=
\left[\phi(x_1), \ldots,  \phi(x_{m-1}) \right]^\top, & k=1, \\[0.1in]
2\mathbf{u}^{(k)}-\mathbf{u}^{(k-1)}, & k>1.
\end{array}
\right.
\]
as {suggested in~\cite{Wang2010}} and
the stopping criterion is
\[
\frac{\| \mathbf{r}_j \|_2}{\| \mathbf{b}^{(k+1)}\|_2} < 10^{-7},
\]
where $\mathbf{r}_j$ is the residual vector after $j$th iteration.

\medskip
\begin{example}\label{Example1}
Consider the equation {\rm (\ref{FDEs})} with $\alpha = 0.8$, $\beta =0.6$, and $\gamma = 1.8$. The left-sided and
right-sided diffusion coefficients are given by
\[
\begin{array}{lcl}
 d_{+}(x,t)= 6(1+t)x^{0.6} , &\qquad&  d_{-}(x,t) = 6(1+t)(1-x)^{0.6},\\[0.12in]
 e_{+}(x,t)= 6(1+t)x^{1.8} , &\qquad&  e_{-}(x,t) = 6(1+t)(1-x)^{1.8}.
 \end{array}
\]
with the spatial interval $ \Omega = (0,1)\times(0,1)$ and the time interval $[0,T] =[0,1]$. The source term and the initial condition are given by
\begin{equation*}
\begin{array}{l}
f(x,t) = e^{t}\bigg[6(1+t)\Big( \big(\frac{\Gamma(4)}{\Gamma(3.4)}
- \frac{\Gamma(4)}{\Gamma(2.2)} \big)\big( x^{3} + (1-x)^{3} \big) -
\big( \frac{3\Gamma(5)}{\Gamma(4.4)} - \frac{3\Gamma(5)}{\Gamma(3.2)} \big) \big( x^{4} + (1-x)^{4} \big)  \\[0.12in]
\qquad\qquad + \big(\frac{3\Gamma(6)}{\Gamma(5.4)}-\frac{3\Gamma(6)}{\Gamma(4.2)}\big) \big( x^{5} + (1-x)^{5} \big) -
\big( \frac{\Gamma(7)}{\Gamma(6.4)}-\frac{\Gamma(7)}{\Gamma(5.2)}\big) \big( x^{6}+ (1-x)^{6}\big) \Big)
+ x^3(1-x)^3 \bigg]
\end{array}
\end{equation*}
and
\begin{equation*}
u(x,0) = x^3(1-x)^3.
\end{equation*}
It can be shown by a direct computation that the solution to the fractional diffusion equation is
\[
u(x,t) = e^{t}x^3(1-x)^3.
\]
\end{example}

\begin{table}[ht]\renewcommand{\arraystretch}{1.1}
\caption{The average number of iterations for Example \ref{Example1}}

\begin{center}
\begin{tabular}{cccccc}
\midrule[1.2pt]
   $m=n$  &  GMRES(20)  &  PGMRES(20) &  CGNR &  PCGNR & error\\
\hline
 $16$ &   8.000  &  3.063   &    12.438 &   3,125  & $4.6312 \times 10^{-4}$ \\
 $32$ &  16.000  &  3.938   &    32.594 &   4.063  & $2.4162 \times 10^{-4}$ \\
 $64$ &  84.969  &  4.063   &   100.547 &   4.813  & $1.3320 \times 10^{-4}$ \\
 $128$ & 231.781  &  5.055   &   318.852 &   4.797  & $7.5522 \times 10^{-5}$ \\
 $256$ & 486.859  &  6.082   &  1060.191 &   5.148  & $4.5765 \times 10^{-5}$ \\
\midrule[1.2pt]
\end{tabular}
\end{center}
\label{t1}
\end{table}

\begin{table}[ht]\renewcommand{\arraystretch}{1.1}
\caption{The required CPU times for Example \ref{Example1}}

\begin{center}
\begin{tabular}{ccccc}
\midrule[1.2pt]
   $m=n$  &  GMRES(20)  &  PGMRES(20) & CGNR & PCGNR \\
\hline
 $16$ &   0.0046    &   0.0310    &   0.0620    &   0.0320  \\
 $32$ &   0.1400    &   0.0470    &   0.2810    &   0.0780  \\
 $64$ &   1.4510    &   0.1560    &   1.7000   &   0.1870  \\
 $128$ &  10.9200    &   0.5930    &  18.2210    &   0.7330  \\
 $256$ &  55.5830    &   3.7120    & 162.7550   &   4.1340  \\
\midrule[1.2pt]
\end{tabular}
\end{center}
\label{t2}
\end{table}

\begin{table}[ht]\renewcommand{\arraystretch}{1.1}
\caption{Condition numbers for relevant matrices for Example \ref{Example1}.}

\begin{center}
\begin{tabular}{ccccccc}
\midrule[1.2pt]
   $m=n$  &  $16$  &  $32$  &  $64$ & $128$ & $256$ \\
\hline
 $k\big(\hat{A}\big)$ &  48.86    & 162.84  &   491.07   &   1.34e+3   &   3.34e+3 \\
  $k\big((P^{(1)}_8 )^{-1}\hat{A}\big)$ & 1.05   & 1.17   &   1.29   &   1.47   &   1.79 \\
 $k\big( \hat{A}^\top\hat{A} \big)$ & 2.39e+3    & 2.65e+4  &  2.41e+5   &  1.79e+6   &   1.16e+7 \\
 $k\big(\big( (P^{(1)}_8)^\top P^{(1)}_8 \big )^{-1}
 \hat{A}
^\top
 \hat{A}
 \big)$ &  1.88 & 20.65 & 193.57  & 960.76 & 3.46e+3 \\
\midrule[1.2pt]
\end{tabular}
\end{center}
\label{t3}
\end{table}

\begin{figure}[ht]\renewcommand{\arraystretch}{1.2}
 \centering
\subfigure[ Spectrum of $\hat{A}$]{\includegraphics[width=0.45\textwidth]{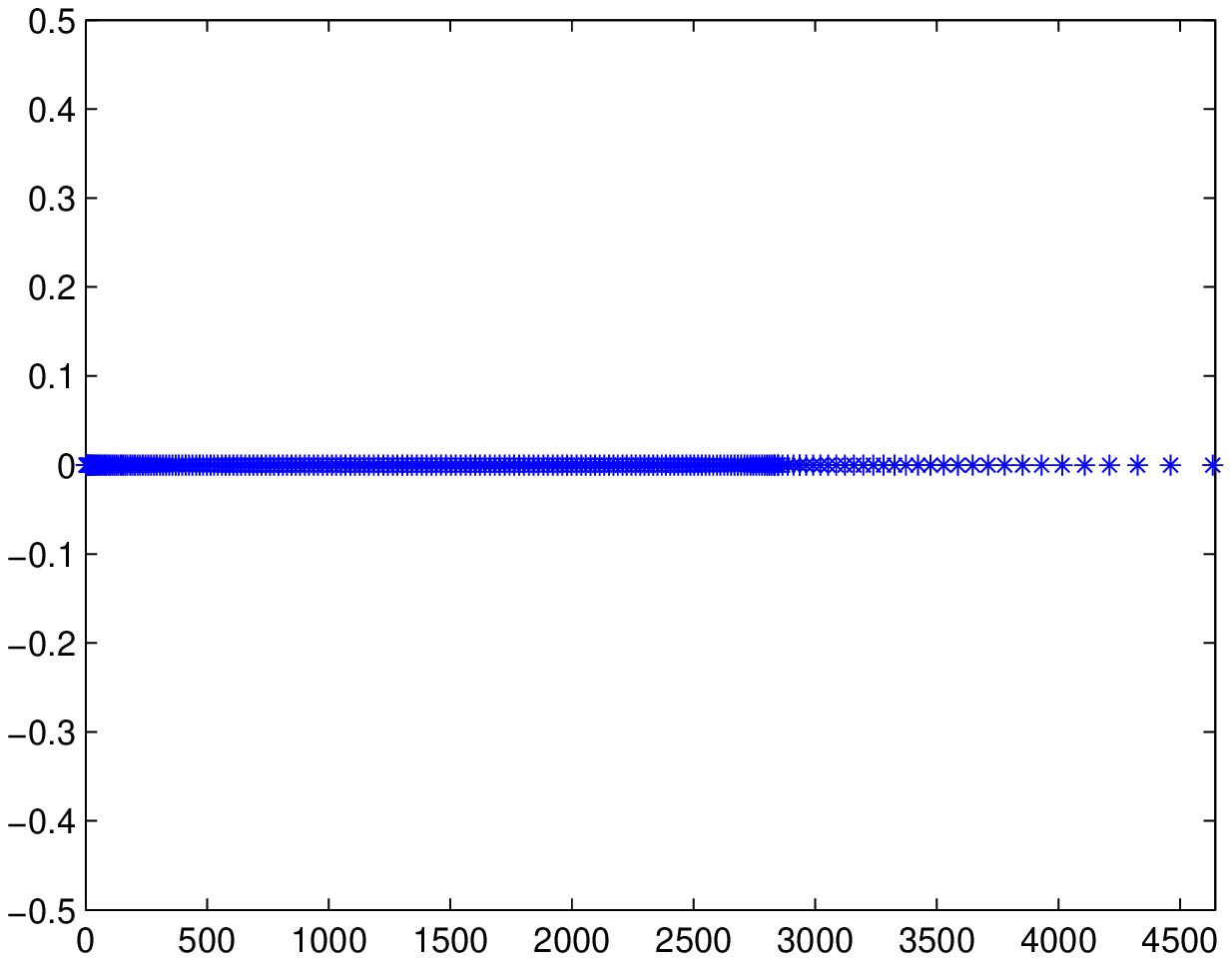}}
\subfigure[ Spectrum of $(P^{(1)}_8 )^{-1}\hat{A}$ ]{\includegraphics[width=0.45\textwidth]{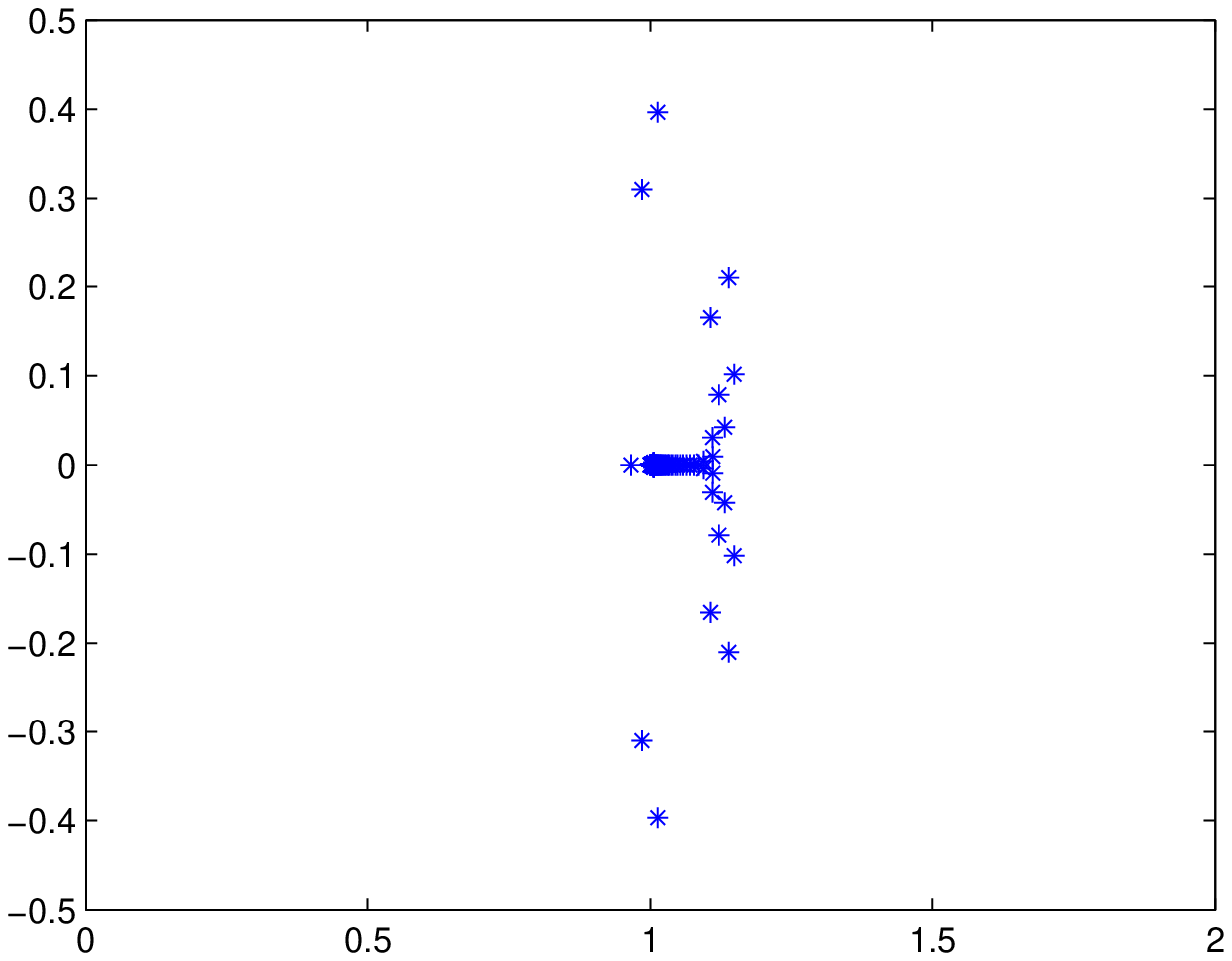}}\\
\subfigure[ Spectrum of $\hat{A}^\top \hat{A}$]{\includegraphics[width=0.45\textwidth]{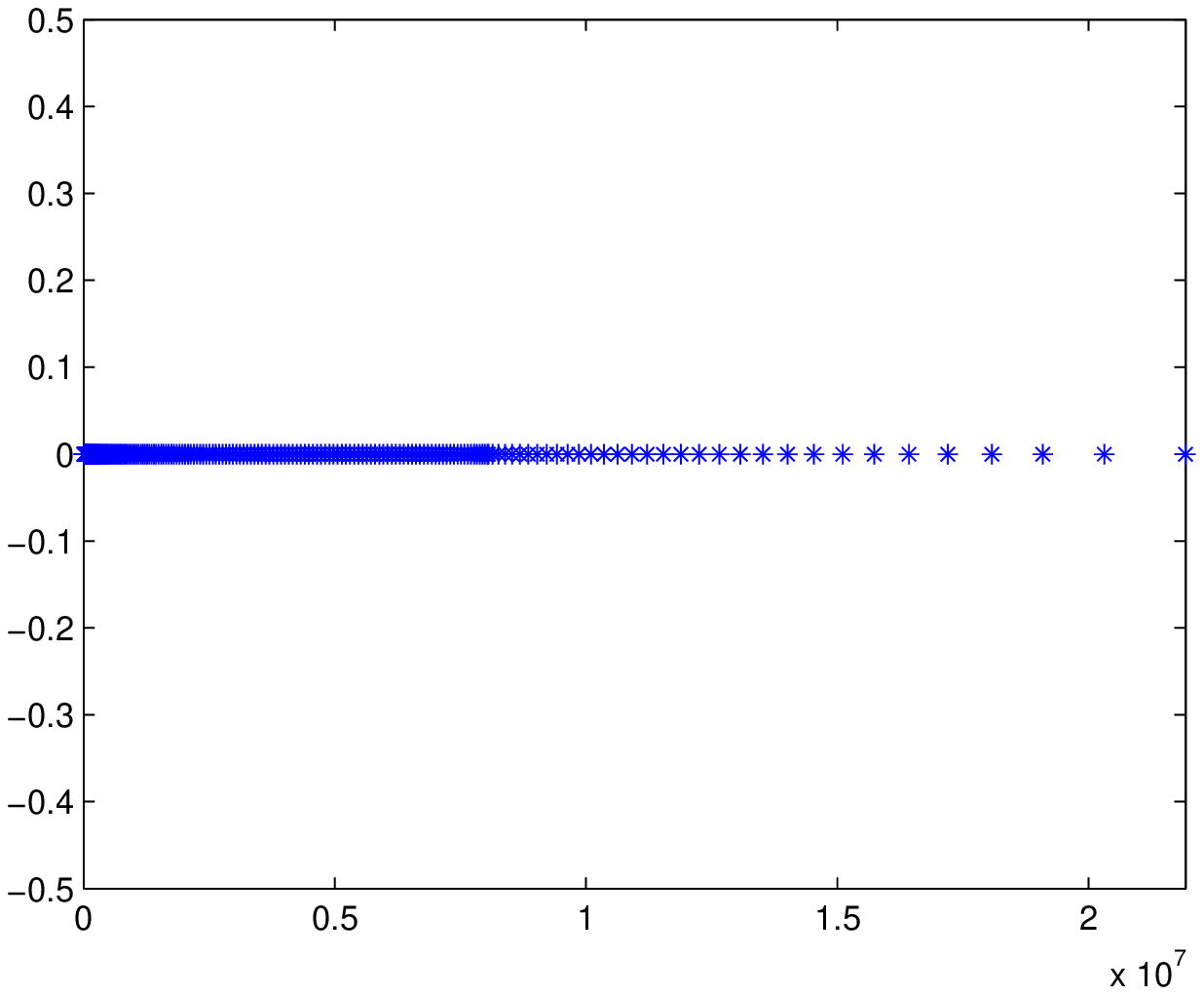}}
\subfigure[ \text{Spectrum of} $\big( (P^{(1)}_8)^\top P^{(1)}_8 \big )^{-1} \hat{A}^\top\hat{A}$]
           {\includegraphics[width=0.45\textwidth]{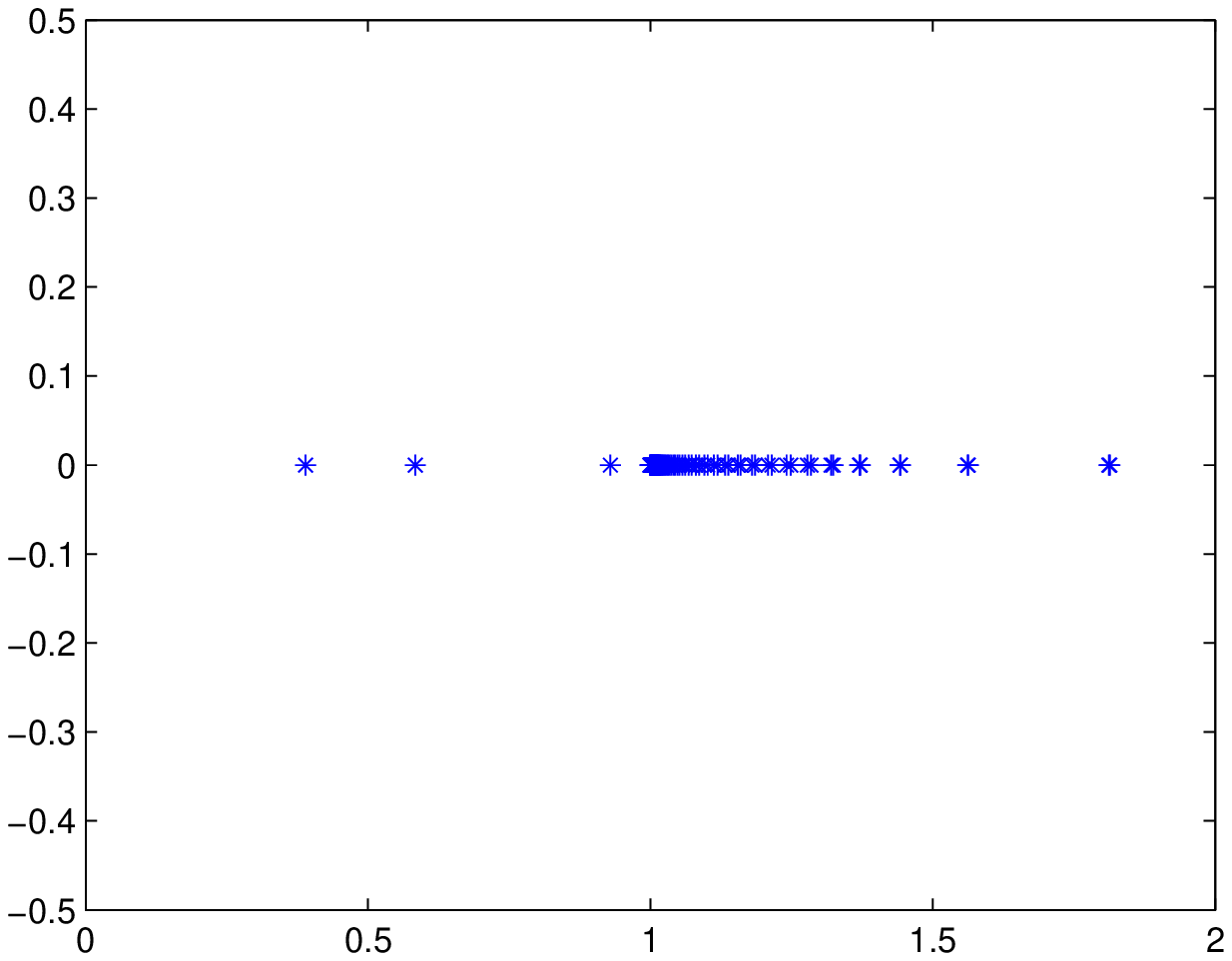}}
 \caption{The Spectra of the unpreconditioned coefficient matrices and the preconditioned coefficient matrices with $m=n=256$.}
 \label{f1}
\end{figure}

The numerical results 
were obtained
by using MATLAB R2010a on a Lenovo Laptop Intel(R) Core(TM)2 Duo of 2.20 GHz CPU and 2GB RAM. We set the bandwidth $\ell$ of the preconditioner $P^{(k+1)}_\ell$ equal to $8$ and use
$``m"$  and $``n" $ to represent the numbers of the spatial partition and the number of the temporal partition, respectively.
%
%
In Tables \ref{t1} and \ref{t2}, we present the average numbers of iterations,
{
the errors computed by the sup-norm between the true solution and the numerical solution
at the last time step}
and the CPU times (seconds) required by GMRES(20), PGMRES(20), CGNR, and PCGNR methods. We see that the number of iterations  and execution time by the GMRES(20) and the CGNR methods increase dramatically, while those by the PGMRES(20) and PCGNR are changed little. The phenomena might be explained by the clustering of the eigenvalues of the relevant coefficient matrices. As an example,
see Figure~\ref{f1} for the distribution of eigenvalues of matrices $\hat{A}$, $\hat{A}^\top \hat{A}$, $(P^{(1)}_8)^{-1}
\hat{A}$ and $\big( (P^{(1)}_8)^\top P^{(1)}_8 \big )^{-1} \hat{A}^\top \hat{A}$ with $m=n=256$.  On the other hand,
we see in Table~\ref{t3} that the effect of the preconditioner on
the condition numbers of the relevant matrices. The reader should be able to notice that the condition number significantly improves with the help of the proposed preconditioner.

\section{Conclusion}

Determining analytic solutions of FDEs is very challenging and remain unknown
for most FDEs. This paper is to present an implicit approach to solve STFDE
with two-sided Gr\"{u}nwald formulae. More significantly, with the aid of~\eqref{eq:FD2},
we can ameliorate the calculation skill by the implementation of efficient and reliable
preconditioning iterative techniques, the PGMRES method and the PCGNR method,
with only computational cost of $\mathcal{O}((m-1)\log (m-1))$.
Numerical results strongly suggest that the efficiency of the proposed preconditioning methods.

\bibliographystyle{elsarticle-num}

\begin{thebibliography}{10}
\expandafter\ifx\csname url\endcsname\relax
  \def\url#1{\texttt{#1}}\fi
\expandafter\ifx\csname urlprefix\endcsname\relax\def\urlprefix{URL }\fi
\expandafter\ifx\csname href\endcsname\relax
  \def\href#1#2{#2} \def\path#1{#1}\fi

\bibitem{Liu2007}
F.~Liu, P.~Zhuang, V.~Anh, I.~Turner, K.~Burrage,
  \href{http://dx.doi.org/10.1016/j.amc.2006.08.162}{Stability and convergence
  of the difference methods for the space-time fractional advection-diffusion
  equation}, Appl. Math. Comput. 191~(1) (2007) 12--20.
\newblock \href {http://dx.doi.org/10.1016/j.amc.2006.08.162}
  {\path{doi:10.1016/j.amc.2006.08.162}}.
\newline\urlprefix\url{http://dx.doi.org/10.1016/j.amc.2006.08.162}

\bibitem{Podlubny1999}
I.~Podlubny, Fractional differential equations, Vol. 198 of Mathematics in
  Science and Engineering, Academic Press, Inc., San Diego, CA, 1999, an
  introduction to fractional derivatives, fractional differential equations, to
  methods of their solution and some of their applications.

\bibitem{Samko1993}
S.~G. Samko, A.~A. Kilbas, O.~I. Marichev, Fractional integrals and
  derivatives, Gordon and Breach Science Publishers, Yverdon, 1993, theory and
  applications, Edited and with a foreword by S. M. Nikol{\cprime}ski{\u\i},
  Translated from the 1987 Russian original, Revised by the authors.

\bibitem{Miller1993}
K.~S. Miller, B.~Ross, An introduction to the fractional calculus and
  fractional differential equations, A Wiley-Interscience Publication, John
  Wiley \& Sons, Inc., New York, 1993.

\bibitem{Oldham1974}
K.~B. Oldham, J.~Spanier, The fractional calculus, Academic Press [A subsidiary
  of Harcourt Brace Jovanovich, Publishers], New York-London, 1974, theory and
  applications of differentiation and integration to arbitrary order, With an
  annotated chronological bibliography by Bertram Ross, Mathematics in Science
  and Engineering, Vol. 111.

\bibitem{Gorenflo2001}
R.~Gorenflo, F.~Mainardi, E.~Scalas, M.~Raberto, Fractional calculus and
  continuous-time finance. {III}. {T}he diffusion limit, in: Mathematical
  finance ({K}onstanz, 2000), Trends Math., Birkh\"auser, Basel, 2001, pp.
  171--180.

\bibitem{Raberto2002}
M.~Raberto, E.~Scalas, F.~Mainardi,
  \href{http://www.sciencedirect.com/science/article/pii/S0378437102010488}{Waiting-times
  and returns in high-frequency financial data: an empirical study}, Physica A:
  Statistical Mechanics and its Applications 314~(1-4) (2002) 749 -- 755,
  horizons in Complex Systems.
\newblock \href
  {http://dx.doi.org/http://dx.doi.org/10.1016/S0378-4371(02)01048-8}
  {\path{doi:http://dx.doi.org/10.1016/S0378-4371(02)01048-8}}.
\newline\urlprefix\url{http://www.sciencedirect.com/science/article/pii/S0378437102010488}

\bibitem{Scalas2000}
E.~Scalas, R.~Gorenflo, F.~Mainardi,
  \href{http://dx.doi.org/10.1016/S0378-4371(00)00255-7}{Fractional calculus
  and continuous-time finance}, Phys. A 284~(1-4) (2000) 376--384.
\newblock \href {http://dx.doi.org/10.1016/S0378-4371(00)00255-7}
  {\path{doi:10.1016/S0378-4371(00)00255-7}}.
\newline\urlprefix\url{http://dx.doi.org/10.1016/S0378-4371(00)00255-7}

\bibitem{Atanackovic2004}
T.~M. Atanackovic, B.~Stankovic,
  \href{http://dx.doi.org.prox.lib.ncsu.edu/10.1088/0305-4470/37/4/012}{On a
  system of differential equations with fractional derivatives arising in rod
  theory}, J. Phys. A 37~(4) (2004) 1241--1250.
\newblock \href {http://dx.doi.org/10.1088/0305-4470/37/4/012}
  {\path{doi:10.1088/0305-4470/37/4/012}}.
\newline\urlprefix\url{http://dx.doi.org.prox.lib.ncsu.edu/10.1088/0305-4470/37/4/012}

\bibitem{Barkai2000}
E.~Barkai, R.~Metzler, J.~Klafter,
  \href{http://dx.doi.org/10.1103/PhysRevE.61.132}{From continuous time random
  walks to the fractional {F}okker-{P}lanck equation}, Phys. Rev. E (3) 61~(1)
  (2000) 132--138.
\newblock \href {http://dx.doi.org/10.1103/PhysRevE.61.132}
  {\path{doi:10.1103/PhysRevE.61.132}}.
\newline\urlprefix\url{http://dx.doi.org/10.1103/PhysRevE.61.132}

\bibitem{Carreras2001}
B.~A. Carreras, V.~E. Lynch, G.~M. Zaslavsky,
  \href{http://scitation.aip.org/content/aip/journal/pop/8/12/10.1063/1.1416180}{Anomalous
  diffusion and exit time distribution of particle tracers in plasma turbulence
  model}, Physics of Plasmas (1994-present) 8~(12) (2001) 5096--5103.
\newblock \href {http://dx.doi.org/http://dx.doi.org/10.1063/1.1416180}
  {\path{doi:http://dx.doi.org/10.1063/1.1416180}}.
\newline\urlprefix\url{http://scitation.aip.org/content/aip/journal/pop/8/12/10.1063/1.1416180}

\bibitem{Kilbas2006}
A.~A. Kilbas, H.~M. Srivastava, J.~J. Trujillo, Theory and applications of
  fractional differential equations, Vol. 204 of North-Holland Mathematics
  Studies, Elsevier Science B.V., Amsterdam, 2006.

\bibitem{Meerschaert2001}
M.~M. Meerschaert, D.~A. Benson, B.~Baeumer,
  \href{http://link.aps.org/doi/10.1103/PhysRevE.63.021112}{Operator l\'evy
  motion and multiscaling anomalous diffusion}, Phys. Rev. E 63 (2001)
  1112--1117.
\newblock \href {http://dx.doi.org/10.1103/PhysRevE.63.021112}
  {\path{doi:10.1103/PhysRevE.63.021112}}.
\newline\urlprefix\url{http://link.aps.org/doi/10.1103/PhysRevE.63.021112}

\bibitem{Meerschaert2002a}
M.~M. Meerschaert, D.~A. Benson, H.-P. Scheffler, B.~Baeumer,
  \href{http://dx.doi.org/10.1103/PhysRevE.65.041103}{Stochastic solution of
  space-time fractional diffusion equations}, Phys. Rev. E (3) 65~(4) (2002)
  041103, 4.
\newblock \href {http://dx.doi.org/10.1103/PhysRevE.65.041103}
  {\path{doi:10.1103/PhysRevE.65.041103}}.
\newline\urlprefix\url{http://dx.doi.org/10.1103/PhysRevE.65.041103}

\bibitem{Meerschaert2002b}
M.~M. Meerschaert, H.-P. Scheffler, Semistable {L}\'evy motion, Fract. Calc.
  Appl. Anal. 5~(1) (2002) 27--54.

\bibitem{Bai2007}
J.~Bai, X.-C. Feng,
  \href{http://dx.doi.org/10.1109/TIP.2007.904971}{Fractional-order anisotropic
  diffusion for image denoising}, IEEE Trans. Image Process. 16~(10) (2007)
  2492--2502.
\newblock \href {http://dx.doi.org/10.1109/TIP.2007.904971}
  {\path{doi:10.1109/TIP.2007.904971}}.
\newline\urlprefix\url{http://dx.doi.org/10.1109/TIP.2007.904971}

\bibitem{West2007}
B.~West, \href{http://dx.doi.org/10.1007/s10955-007-9294-0}{Fractional calculus
  in bioengineering}, Journal of Statistical Physics 126~(6) (2007) 1285--1286.
\newblock \href {http://dx.doi.org/10.1007/s10955-007-9294-0}
  {\path{doi:10.1007/s10955-007-9294-0}}.
\newline\urlprefix\url{http://dx.doi.org/10.1007/s10955-007-9294-0}

\bibitem{Ervin2007}
V.~J. Ervin, N.~Heuer, J.~P. Roop,
  \href{http://dx.doi.org/10.1137/050642757}{Numerical approximation of a time
  dependent, nonlinear, space-fractional diffusion equation}, SIAM J. Numer.
  Anal. 45~(2) (2007) 572--591.
\newblock \href {http://dx.doi.org/10.1137/050642757}
  {\path{doi:10.1137/050642757}}.
\newline\urlprefix\url{http://dx.doi.org/10.1137/050642757}

\bibitem{Ford2001}
N.~J. Ford, A.~C. Simpson, \href{http://dx.doi.org/10.1023/A:1016601312158}{The
  numerical solution of fractional differential equations: speed versus
  accuracy}, Numer. Algorithms 26~(4) (2001) 333--346.
\newblock \href {http://dx.doi.org/10.1023/A:1016601312158}
  {\path{doi:10.1023/A:1016601312158}}.
\newline\urlprefix\url{http://dx.doi.org/10.1023/A:1016601312158}

\bibitem{Liu2004}
F.~Liu, V.~Anh, I.~Turner,
  \href{http://dx.doi.org/10.1016/j.cam.2003.09.028}{Numerical solution of the
  space fractional {F}okker-{P}lanck equation}, J. Comput. Appl. Math. 166~(1)
  (2004) 209--219.
\newblock \href {http://dx.doi.org/10.1016/j.cam.2003.09.028}
  {\path{doi:10.1016/j.cam.2003.09.028}}.
\newline\urlprefix\url{http://dx.doi.org/10.1016/j.cam.2003.09.028}

\bibitem{Meerschaert2004}
M.~M. Meerschaert, C.~Tadjeran,
  \href{http://dx.doi.org/10.1016/j.cam.2004.01.033}{Finite difference
  approximations for fractional advection-dispersion flow equations}, J.
  Comput. Appl. Math. 172~(1) (2004) 65--77.
\newblock \href {http://dx.doi.org/10.1016/j.cam.2004.01.033}
  {\path{doi:10.1016/j.cam.2004.01.033}}.
\newline\urlprefix\url{http://dx.doi.org/10.1016/j.cam.2004.01.033}

\bibitem{Meerschaert2006a}
M.~M. Meerschaert, C.~Tadjeran,
  \href{http://dx.doi.org.prox.lib.ncsu.edu/10.1016/j.apnum.2005.02.008}{Finite
  difference approximations for two-sided space-fractional partial differential
  equations}, Appl. Numer. Math. 56~(1) (2006) 80--90.
\newblock \href {http://dx.doi.org/10.1016/j.apnum.2005.02.008}
  {\path{doi:10.1016/j.apnum.2005.02.008}}.
\newline\urlprefix\url{http://dx.doi.org.prox.lib.ncsu.edu/10.1016/j.apnum.2005.02.008}

\bibitem{Meerschaert2006b}
M.~M. Meerschaert, H.-P. Scheffler, C.~Tadjeran,
  \href{http://dx.doi.org/10.1016/j.jcp.2005.05.017}{Finite difference methods
  for two-dimensional fractional dispersion equation}, J. Comput. Phys. 211~(1)
  (2006) 249--261.
\newblock \href {http://dx.doi.org/10.1016/j.jcp.2005.05.017}
  {\path{doi:10.1016/j.jcp.2005.05.017}}.
\newline\urlprefix\url{http://dx.doi.org/10.1016/j.jcp.2005.05.017}

\bibitem{Sousa2009}
E.~Sousa, \href{http://dx.doi.org/10.1016/j.jcp.2009.02.011}{Finite difference
  approximations for a fractional advection diffusion problem}, J. Comput.
  Phys. 228~(11) (2009) 4038--4054.
\newblock \href {http://dx.doi.org/10.1016/j.jcp.2009.02.011}
  {\path{doi:10.1016/j.jcp.2009.02.011}}.
\newline\urlprefix\url{http://dx.doi.org/10.1016/j.jcp.2009.02.011}

\bibitem{Tadjeran2006}
C.~Tadjeran, M.~M. Meerschaert, H.-P. Scheffler,
  \href{http://dx.doi.org.prox.lib.ncsu.edu/10.1016/j.jcp.2005.08.008}{A
  second-order accurate numerical approximation for the fractional diffusion
  equation}, J. Comput. Phys. 213~(1) (2006) 205--213.
\newblock \href {http://dx.doi.org/10.1016/j.jcp.2005.08.008}
  {\path{doi:10.1016/j.jcp.2005.08.008}}.
\newline\urlprefix\url{http://dx.doi.org.prox.lib.ncsu.edu/10.1016/j.jcp.2005.08.008}

\bibitem{Tadjeran2007}
C.~Tadjeran, M.~M. Meerschaert,
  \href{http://dx.doi.org/10.1016/j.jcp.2006.05.030}{A second-order accurate
  numerical method for the two-dimensional fractional diffusion equation}, J.
  Comput. Phys. 220~(2) (2007) 813--823.
\newblock \href {http://dx.doi.org/10.1016/j.jcp.2006.05.030}
  {\path{doi:10.1016/j.jcp.2006.05.030}}.
\newline\urlprefix\url{http://dx.doi.org/10.1016/j.jcp.2006.05.030}

\bibitem{Wang2010}
H.~Wang, K.~Wang, T.~Sircar,
  \href{http://dx.doi.org.prox.lib.ncsu.edu/10.1016/j.jcp.2010.07.011}{A direct
  {$O(N\log^2N)$} finite difference method for fractional diffusion equations},
  J. Comput. Phys. 229~(21) (2010) 8095--8104.
\newblock \href {http://dx.doi.org/10.1016/j.jcp.2010.07.011}
  {\path{doi:10.1016/j.jcp.2010.07.011}}.
\newline\urlprefix\url{http://dx.doi.org.prox.lib.ncsu.edu/10.1016/j.jcp.2010.07.011}

\bibitem{Chan1996}
R.~H. Chan, M.~K. Ng,
  \href{http://dx.doi.org/10.1137/S0036144594276474}{Conjugate gradient methods
  for {T}oeplitz systems}, SIAM Rev. 38~(3) (1996) 427--482.
\newblock \href {http://dx.doi.org/10.1137/S0036144594276474}
  {\path{doi:10.1137/S0036144594276474}}.
\newline\urlprefix\url{http://dx.doi.org/10.1137/S0036144594276474}

\bibitem{Ng2004}
M.~K. Ng, Iterative methods for {T}oeplitz systems, Numerical Mathematics and
  Scientific Computation, Oxford University Press, New York, 2004.

\bibitem{Chan2007}
R.~H. Chan, X.-Q. Jin,
  \href{http://dx.doi.org.prox.lib.ncsu.edu/10.1137/1.9780898718850}{An
  introduction to iterative {T}oeplitz solvers}, Vol.~5 of Fundamentals of
  Algorithms, Society for Industrial and Applied Mathematics (SIAM),
  Philadelphia, PA, 2007.
\newblock \href {http://dx.doi.org/10.1137/1.9780898718850}
  {\path{doi:10.1137/1.9780898718850}}.
\newline\urlprefix\url{http://dx.doi.org.prox.lib.ncsu.edu/10.1137/1.9780898718850}

\bibitem{Wang2011}
K.~Wang, H.~Wang,
  \href{http://www.sciencedirect.com/science/article/pii/S0309170810001983}{A
  fast characteristic finite difference method for fractional
  advection-diffusion equations}, Advances in Water Resources 34~(7) (2011) 810
  -- 816.
\newblock \href
  {http://dx.doi.org/http://dx.doi.org/10.1016/j.advwatres.2010.11.003}
  {\path{doi:http://dx.doi.org/10.1016/j.advwatres.2010.11.003}}.
\newline\urlprefix\url{http://www.sciencedirect.com/science/article/pii/S0309170810001983}

\bibitem{Lei2013}
S.-L. Lei, H.-W. Sun, \href{http://dx.doi.org/10.1016/j.jcp.2013.02.025}{A
  circulant preconditioner for fractional diffusion equations}, J. Comput.
  Phys. 242 (2013) 715--725.
\newblock \href {http://dx.doi.org/10.1016/j.jcp.2013.02.025}
  {\path{doi:10.1016/j.jcp.2013.02.025}}.
\newline\urlprefix\url{http://dx.doi.org/10.1016/j.jcp.2013.02.025}

\bibitem{Lin2014}
F.-R. Lin, S.-W. Yang, X.-Q. Jin,
  \href{http://dx.doi.org.prox.lib.ncsu.edu/10.1016/j.jcp.2013.07.040}{Preconditioned
  iterative methods for fractional diffusion equation}, J. Comput. Phys. 256
  (2014) 109--117.
\newblock \href {http://dx.doi.org/10.1016/j.jcp.2013.07.040}
  {\path{doi:10.1016/j.jcp.2013.07.040}}.
\newline\urlprefix\url{http://dx.doi.org.prox.lib.ncsu.edu/10.1016/j.jcp.2013.07.040}

\bibitem{Saad2003}
Y.~Saad, \href{http://dx.doi.org/10.1137/1.9780898718003}{Iterative methods for
  sparse linear systems}, 2nd Edition, Society for Industrial and Applied
  Mathematics, Philadelphia, PA, 2003.
\newblock \href {http://dx.doi.org/10.1137/1.9780898718003}
  {\path{doi:10.1137/1.9780898718003}}.
\newline\urlprefix\url{http://dx.doi.org/10.1137/1.9780898718003}

\bibitem{Saad1986}
Y.~Saad, M.~H. Schultz, \href{http://dx.doi.org/10.1137/0907058}{G{MRES}: a
  generalized minimal residual algorithm for solving nonsymmetric linear
  systems}, SIAM J. Sci. Statist. Comput. 7~(3) (1986) 856--869.
\newblock \href {http://dx.doi.org/10.1137/0907058}
  {\path{doi:10.1137/0907058}}.
\newline\urlprefix\url{http://dx.doi.org/10.1137/0907058}

\bibitem{Pang2012}
H.-K. Pang, H.-W. Sun,
  \href{http://dx.doi.org.prox.lib.ncsu.edu/10.1016/j.jcp.2011.10.005}{Multigrid
  method for fractional diffusion equations}, J. Comput. Phys. 231~(2) (2012)
  693--703.
\newblock \href {http://dx.doi.org/10.1016/j.jcp.2011.10.005}
  {\path{doi:10.1016/j.jcp.2011.10.005}}.
\newline\urlprefix\url{http://dx.doi.org.prox.lib.ncsu.edu/10.1016/j.jcp.2011.10.005}

\bibitem{Demmel1997}
J.~W. Demmel, \href{http://dx.doi.org/10.1137/1.9781611971446}{Applied
  numerical linear algebra}, Society for Industrial and Applied Mathematics
  (SIAM), Philadelphia, PA, 1997.
\newblock \href {http://dx.doi.org/10.1137/1.9781611971446}
  {\path{doi:10.1137/1.9781611971446}}.
\newline\urlprefix\url{http://dx.doi.org/10.1137/1.9781611971446}

\end{thebibliography}
\def\cprime{$'$}

\end{document}